\documentclass[11pt,reqno]{amsart}

\usepackage[T1]{fontenc}
\usepackage[utf8]{inputenc}
\usepackage[english]{babel}
\usepackage{microtype}
\usepackage{fancyhdr}

\usepackage{amssymb}
\usepackage{mathrsfs}
\usepackage{stmaryrd}
\usepackage[inline]{enumitem}
\usepackage{xcolor}
\usepackage{comment}
\PassOptionsToPackage{hyphens}{url}\usepackage{hyperref}
\usepackage{aliascnt}
\usepackage[nameinlink,capitalise,noabbrev]{cleveref}

\hypersetup{
  colorlinks=true,
  linkcolor=blue!55!black,
  citecolor=blue!55!black,
  urlcolor=blue!55!black,
  pdftitle={Power Semigroups and Rigidity Theorems for Groups},
  pdfauthor={Liu Shuolin and Salvatore Tringali}
}

\theoremstyle{definition}

\newtheorem{theorem}{Theorem}[section]
\newtheorem{proposition}[theorem]{Proposition}

\newtheorem{questions}[theorem]{Questions}
\newtheorem{lemma}[theorem]{Lemma}

\newtheorem{definition}[theorem]{Definition}

\theoremstyle{remark}

\newtheorem{remark}[theorem]{Remark}
\newtheorem{remarks}[theorem]{Remarks}

\newtheoremstyle{theoremdd}% name of the style to be used
  {\topsep}% measure of space to leave above the theorem. E.g.: 3pt
  {\topsep}% measure of space to leave below the theorem. E.g.: 3pt
  {\normalfont}% name of font to use in the body of the theorem
  {0pt}% measure of space to indent
  {\scshape}% name of head font
  {:}% punctuation between head and body
  { }% space after theorem head; " " = normal interword space
  {\indent\thmname{#1}\thmnumber{ #2}\textnormal{\thmnote{ (#3)}}}

\theoremstyle{theoremdd}

\newcommand{\Pfin}{\mathcal P_{\mathrm{fin}}}
\providecommand\llb{\llbracket}
\providecommand\rrb{\rrbracket}
\newcommand{\evid}[1]{\textsf{#1}}
\newcommand{\fin}{\mathrm{fin}}
\providecommand\llb{\llbracket}
\providecommand\rrb{\rrbracket}

\renewcommand{\emptyset}{\varnothing}
\renewcommand{\setminus}{\smallsetminus}
\renewcommand{\,}{\kern 0.1em}

\title[Power Semigroups and Two Rigidity Theorems for Groups]
{Power Semigroups and Two Rigidity Theorems for Groups}

\author{Shuolin Liu}
\address{(S.L.) School of Mathematical Sciences, Hebei Normal University | Shijiazhuang, Hebei Province, 050024, China}
\email{liushuolin77@gmail.com}

\author{Salvatore Tringali}
\address{(S.T.) School of Mathematical Sciences, Hebei Normal University | Shijiazhuang, Hebei Province, 050024, China}
\email{salvo.tringali@gmail.com}
\urladdr{https://salvo-tringali.github.io/home/}



\begin{document}

\begin{abstract}
Let $\mathcal P(H)$ be the semigroup obtained by endowing the family of all non-empty subsets of a semigroup $H$ with the setwise operation naturally induced by $H$ on its power set, and denote by $\mathcal P_\text{fin}(H)$ the subsemigroup of $\mathcal P(H)$ consisting of all non-empty finite subsets of $H$.

We obtain (as a corollary of a theorem of independent interest) that if $H$ is a group and $K$ is a semigroup, then $\mathcal P(H) \cong \mathcal P(K)$ implies $H \cong K$. The finitary analogue of this statement is considerably more difficult, and we prove it only for $H$ an additive subgroup of the rationals. 
Most notably, the proof of the second result relies, in a rather circuitous way, on a special case of the Evertse--Schlickewei--Schmidt theorem.
\end{abstract}
\maketitle

% -----------------------------------------------------------------------------
\section{Introduction}
\label{sect:1}
Let $H$ be a semigroup (see the end of the section for notation and terminology). The family of all non-empty subsets of $H$ is itself a semigroup, here denoted by $\mathcal P(H)$ and referred to as the 
\evid{large power semigroup} of $H$, when endowed with the binary operation of 
setwise multiplication induced by $H$ on its power set and defined by 
\[
AB := \{ab: a \in A,\, b \in B\}, \qquad\text{for all } A, B \subseteq H.
\]
Moreover, the family $\Pfin(H)$ of all non-empty \textit{finite} subsets of $H$ is a subsemigroup of $\mathcal P(H)$, called the \evid{finitary power semigroup} of $H$. These structures are generically dubbed \evid{power semigroups}.

If the semigroup $H$ is written additively, we adopt the same convention for any subsemigroup of $\mathcal P(H)$. This amounts, in practice, to the fact that the operation on
$\mathcal P(H)$ maps a pair $(A, B)$ of non-empty subsets of $H$ to their \evid{sumset}
\[
A + B := \{a + b: a \in A,\, b \in B\},
\]
as opposed to the setwise product of $A$ by $B$ used in the multiplicative setting.

The systematic investigation of power semigroups began with the work of Tamura and Shafer~\cite{Tam-Sha-1967} in the late 1960s, although Bailieu~\cite{bal-1950} had already studied the subgroups of the large power semigroup of a group in 1950. Tamura, in particular, maintained a lifelong interest in the following problem (here and later, $\cong$ denotes semigroup isomorphism).

\begin{questions}
\label{que:tamura-shafer}
Let $\mathcal O$ be a class of semigroups. Prove or disprove that, for all $H, K \in \mathcal O$, 
\begin{enumerate}[label=\textup{(\arabic{*})}]
\item\label{que:tamura-shafer(1)}
$\mathcal P(H) \cong \mathcal P(K)$ if and only if $H \cong K$.
\item\label{que:tamura-shafer(2)} $\mathcal P_\fin(H) \cong \mathcal P_\fin(K)$ if and only if $H \cong K$.
\end{enumerate}
\end{questions}

A class $\mathcal{O}$ for which Question~\ref{que:tamura-shafer}\ref{que:tamura-shafer(1)} has a positive answer is sometimes said to be \evid{globally determined}. This terminology stems from the fact that an isomorphism from the large power 
semigroup of a semigroup $H$ to that of a semigroup $K$ is often referred to as a \evid{global isomorphism} from $H$ to $K$ (by the same token, some 
authors call $\mathcal{P}(H)$ the \evid{global} of $H$).

While the class of \emph{all} semigroups is not globally 
determined~\cite{Mog-1973}, many interesting subclasses are,  including groups \cite{Sha-1967},
semilattices \cite[p.~218]{Kob-1984}, Clifford semigroups
\cite[Theorem~4.7]{Gan-Zha-2014}, and cancellative \emph{commutative} semigroups \cite[Corollary~1]{Tri-2025(c)}. Question \ref{que:tamura-shafer}\ref{que:tamura-shafer(1)} remains open for \textit{finite} semigroups (despite some authors having claimed otherwise):

\begin{center}
\url{https://mathoverflow.net/questions/508548/}
\end{center}

As for Question \ref{que:tamura-shafer}\ref{que:tamura-shafer(2)}, it appears to have been first formulated as Question~4(1) in \cite{GarSan-Tri2025(a)}, and little is known beyond the case where $H$ is finite, or equivalently $\mathcal{P}_{\fin}(H) = \mathcal{P}(H)$.

After a period of intense activity in the 1970s--1990s, power
semigroups entered a phase of relative dormancy until they were
independently rediscovered in a 2018 paper by Fan and Tringali
\cite{Fa-Tr18}. Since then, the subject has gained renewed momentum, with a number of papers addressing a broad spectrum of questions, from the study of automorphism groups
\cite{Tri-Yan2025(b), Tri-Wen-2026(a), Rago26(b), Tri-Wen-2026(b), Wong2026}
to new kinds of isomorphism problems
\cite{Bie-Ger-22, Tri-Yan2025(a), GarSan-Tri2025(a), Rago26,
Tri-Yan2026(a), Rago26(c)} and the investigation of arithmetic properties
\cite{An-Tr18, Bie-Ger-22, 
Gonz-Li-Rabi-Rodr-Tira-2025, Cos-Tri2025}. We refer the reader to \cite{Tri-Survey} for a survey of these
recent developments, while noting that significant contributions to some of these lines of research had already been made by Byrd et al.~\cite{Byr-Llo-Ped-Ste-1977, Byr-Llo-Men-Tel-1978, Byr-Llo-Ste-1982, Byr-Llo-Ped-Ste-1984} in the late 1970s and early 1980s.

In the present paper, we focus on two problems that
have not yet received much attention in the literature, despite
being quite natural and closely related to
Questions~\ref{que:tamura-shafer}.

\begin{questions}
\label{que:2}
Let $\mathcal O$ be a class of semigroups that is closed under isomorphisms, and let $K$ be an arbitrary semigroup. Prove or disprove that 
\begin{enumerate}[label=\textup{(\arabic{*})}]
\item\label{que:2(1)} if $H \in \mathcal O$ and $\mathcal P(H) \cong \mathcal P(K)$, then $K \in \mathcal O$.
\item\label{que:2(2)} if $H \in \mathcal O$ and $\mathcal P_\fin(H) \cong \mathcal P_\fin(K)$, then $K \in \mathcal O$.
\end{enumerate}
\end{questions}

We show in Section~\ref{sect:2} that the answer to
Question~\ref{que:2}\ref{que:2(1)} is positive when
$\mathcal O$ is the class of groups. More precisely, we establish the following ``rigidity theorem''.

\begin{theorem}
\label{thm:groups-are-P-closed}
If $H$ is a group and $K$ is a semigroup with $\mathcal P(H) \cong \mathcal P(K)$, then $H \cong K$.
\end{theorem}

The proof of Theorem \ref{thm:groups-are-P-closed} is deferred to the end of Section \ref{sect:2}. While elementary in hindsight, it arises as a byproduct of a more general result of independent interest for the study of isomorphism problems on power semigroups: every global isomorphism from a monoid $H$ to a monoid $K$ restricts to a global isomorphism from the unit group of $H$ to the unit group of $K$ (Theorem~\ref{thm:global-iso-maps-unit-group-to-unit-group}).

Since $\mathcal P(H) = \mathcal P_\fin(H)$ for any finite semigroup $H$, Questions \ref{que:2}\ref{que:2(1)} and \ref{que:2}\ref{que:2(2)} coincide for finite groups. Yet, the specialization of the latter to \textit{infinite} groups appears to be much more difficult, and here we only settle it in the fundamental case where $\mathcal O$ is the class of all semigroups isomorphic to a subgroup of $\mathbb Q$. This leads to our main result, which will be proved in Section~\ref{sect:4}.

\begin{theorem}\label{thm:main}
Let $H$ be an additive subgroup of the rational numbers and $K$ be a semigroup. If $\Pfin(H) \cong \Pfin(K)$, then $H \cong K$.
\end{theorem}

Additive subgroups of $\mathbb{Q}$ are completely classified by Baer's classical work \cite{Bae-1937}. For instance, the theorem applies when
$H$ is $\mathbb Q$ itself, $\mathbb Z$, or, for a prime $p$, the group
\[
\mathbb Z[1/p] := \{a/p^n: a \in \mathbb Z \text{ and } n \in \mathbb N\}.
\]

The proof begins by establishing that $K \cong \allowbreak L \times H^\times$ for a certain monoid $L$ with trivial unit group (Proposition \ref{prop:structural-splitting}). The goal is then to demonstrate that $L$ is itself trivial, which is achieved in two steps. First, we find that, for every $n \in \mathbb{N}^+$, the number of sets 
$X \in \mathcal{P}_{\text{fin}}(\mathbb N)$ such that
\[
\{0, 1\} + X = \{0, 1, \ldots, n\}
\]
is the $n$-th Fibonacci number~$F_n$ (Proposition~\ref{prop:fib-count}). 
Second, we prove that, if $L$ is non-trivial, then there exist infinitely 
many~$n$ for which~$F_n$ admits a representation 
of the form $\sum_{i=1}^k 2^{e_i} \varepsilon_i$ over $\mathbb Z$, where the number $k$ 
of summands is independent of $n$, the exponents $e_i$ are non-negative 
integers, and the coefficients $\varepsilon_i$ are either $1$ or $-1$ (Definition \ref{def:indexed-signed-bin-rep}). However, 
this will be seen to be impossible (Lemma~\ref{lem:fib-binary-representations}) 
as a consequence of the following special case of a celebrated result by Evertse, Schlickewei, 
and Schmidt~\cite[Theorem 1.1]{ESS2002}.

\begin{theorem}\label{thm:ess}
If $\Gamma$ is a finitely generated subgroup of $\mathbb C^\times$, then for every $r \in \mathbb N^+$ the equation
\begin{equation}
\label{thm:ess:eq(1)}
x_1 + \cdots + x_r = 1
\end{equation}
has only finitely many non-degenerate solutions $(x_1, \ldots, x_r)$ with $x_1, \ldots, x_r \in \Gamma$.
\end{theorem}

Here, a solution $(x_1, \ldots, x_r)$ of Eq.~\eqref{thm:ess:eq(1)} is
\evid{non-degenerate} if $\sum_{i \in I} x_i \ne 0$ for every non-empty proper subset $I$ of the interval $\llb 1, r \rrb$.

\subsection*{Generalities} 
We denote by $\mathbb N$ the additive monoid of non-negative integers, by $\mathbb N^+$ the set of positive integers, by $\mathbb Z$ the additive group of integers, by $\mathbb Q$ the additive group of rationals, by $\mathbb C$ the complex field, by $|X|$ the cardinality of a set $X$, and by $2^X$ the power set of $X$.

Unless stated otherwise, we reserve the letters $i$ and $k$ (with or without subscripts) for non-negative integers, and the letters $m$ and $n$ for positive integers.
Given $a, b \in \mathbb N$, we let $\llb a, b \rrb := \{x \in \allowbreak \mathbb N \colon \allowbreak a \le x \le b\}$ be the (\evid{discrete}) \evid{interval} from $a$ to $b$.

We address the reader to \cite{Ho95} for basic aspects of semigroup theory. If not explicitly specified, we write all semigroups multiplicatively. Accordingly, we use the symbol $1_H$ for the identity (element) of a monoid $H$.
A \evid{unit} of $H$ is then an element $u \in H$ for which there exists a provably unique element $v \in H$, denoted by $u^{-1}$ and called the \evid{inverse} of $u$ (in $H$), with the property that $uv = vu = 1_H$. The set of all units of $H$ is a subgroup of $H$, denoted by $H^\times$ and referred to as the \evid{unit group} of $H$. A monoid is \evid{trivial} if its only element is the identity. If $R$ is a unital ring, we write $R^\times$ for the unit group of its multiplicative monoid.

Further notation and ter\-mi\-nol\-o\-gy, if not explained when first used, are standard or should be clear from the context. Most notably, all morphisms considered through this paper are \textit{semigroup} homomorphisms (unless otherwise stated), and for a function $f: A \to B$ and a set $X \subseteq A$ we let $f[X]$ be the \evid{direct image} $\{f(x): x \in X\}$ of $X$ under $f$.
\section{The infinitary setting}
\label{sect:2}

In this section, we provide a proof of Theorem \ref{thm:groups-are-P-closed}. In the process, we derive a result on global isomorphisms (Theorem \ref{thm:global-iso-maps-unit-group-to-unit-group}) that may be of independent interest.

\begin{definition}
\label{def:unit-stable-element}
Given a monoid $H$, we say that an element $x \in H$ is \evid{unit-stable} if $x = ux = xu$ for all $u \in H^\times$, where $H^\times$ denotes the unit group of $H$. 
\end{definition}

Every element of a monoid with trivial unit group is unit-stable, whereas no element of a non-trivial group is unit-stable. It is this contrast in behavior, together with the following remarks, that underlies the proof of Theorem \ref{thm:groups-are-P-closed}.

\begin{remarks}
\label{remarks:units-and-unit-stability}
\begin{enumerate*}[label=\textup{(\arabic{*})}, mode=unboxed]
\item\label{remarks:units-and-unit-stability(1)} Let $H$ be a monoid and $X$ be a non-empty subset of $H$. Since $1_H \in H^\times$ and hence $X = \allowbreak X 1_H \subseteq XH^\times$, it is clear that $XH^\times = H^\times$ yields $X \subseteq H^\times$. On the other hand, we have $uH^\times = \allowbreak H^\times$ for every $u \in H^\times$. Therefore, if $X \subseteq H^\times$, then 
\[
XH^\times = \bigcup_{u \in X} uH^\times = H^\times.
\]
By symmetry, this proves that $X \subseteq H^\times$ if and only if $XH^\times = H^\times$, if and only if $H^\times X = H^\times$.
\end{enumerate*}

\vskip 0.05cm

\begin{enumerate*}[label=\textup{(\arabic{*})}, resume, mode=unboxed]
\item\label{remarks:units-and-unit-stability(2)} 
The units of the large power monoid $\mathcal{P}(H)$ of a monoid $H$ are precisely 
the singletons $\{u\}$ with $u \in H^\times$. 
In particular, if $U, V \in \mathcal{P}(H)$ are such that $UV = VU = \{1_H\}$, 
then $uv = vu = 1_H$ for all $u \in U$ and $v \in V$. Therefore, $U$ and $V$ 
are subsets of $H^\times$. This shows that every element of $V$ is cancellative, so that $|U| \le |UV| = 1$ and hence $U = \{u\}$ for some $u \in H^\times$.\\

\indent{}It follows that a non-empty subset $X \subseteq H$ is unit-stable as an element of 
$\mathcal{P}(H)$ if and only if 
\[
uX = Xu = X, \qquad \text{for all } u \in H^\times,
\]
which is equivalent to 
\[
H^\times X = \bigcup_{u \in H^\times} uX = X. 
\]
Conversely, if $H^\times X = X$ and $u \in H^\times$, then 
\[
uX \subseteq H^\times X = X = u u^{-1}X \subseteq u H^\times X = uX.
\]
By symmetry, we conclude that $X$ is unit-stable in 
$\mathcal{P}(H)$ if and only if $
H^\times X = XH^\times = X$.
\end{enumerate*}

\vskip 0.05cm

\begin{enumerate*}[label=\textup{(\arabic{*})}, resume, mode=unboxed]
\item\label{remarks:units-and-unit-stability(3)} Let $f$ be a semigroup isomorphism from a monoid $H$ to a monoid $K$. 
It is a basic fact that $f$ sends the identity $1_H$ of $H$ to the identity $1_K$ of $K$ 
(see, e.g., the last lines of \cite[Sect.~2]{Tri-2025(c)}). Hence, $f$ restricts to an isomorphism from $H^\times$ to $K^\times$, 
and it is then routine to check that it preserves unit-stable elements.
\end{enumerate*}
\end{remarks}

We are now ready to prove the main theorem of this section, along with Theorem \ref{thm:groups-are-P-closed}.

\begin{theorem}
\label{thm:global-iso-maps-unit-group-to-unit-group}
The following hold for a global isomorphism $f$ from a monoid $H$ to a monoid $K$:
\begin{enumerate}[label=\textup{(\roman{*})}]
\item\label{prop:global-iso-maps-unit-group-to-unit-group(1)} 
$f(H^\times) = K^\times$.
\item\label{prop:global-iso-maps-unit-group-to-unit-group(2)} $f$ restricts to a global isomorphism from $H^\times$ to $K^\times$.
\end{enumerate}
\end{theorem}

\begin{proof}
\ref{prop:global-iso-maps-unit-group-to-unit-group(1)} Let $\Omega(M)$ be the set of unit-stable elements of the large power monoid $\mathcal P(M)$ of a monoid $M$. Since the inverse $f^{-1}$ of $f$ is a global isomorphism from $K$ to $H$, it is clear from Remark \ref{remarks:units-and-unit-stability}\ref{remarks:units-and-unit-stability(3)} that $f[\Omega(H)] = \Omega(K)$. On the other hand, $M^\times$ is a unit-stable element of $\mathcal P(M)$, as guaranteed by Remark \ref{remarks:units-and-unit-stability}\ref{remarks:units-and-unit-stability(2)} when noting that $M^\times M^\times = M^\times$.

As a result, $f(H^\times) \in \Omega(K)$, which, again by Remark \ref{remarks:units-and-unit-stability}\ref{remarks:units-and-unit-stability(2)}, implies $f(H^\times) K^\times = f(H^\times)$. Likewise, $H^\times X = X$, where $X := f^{-1}(K^\times) \in \Omega(H)$. Therefore, 
\[
K^\times = f(X) = f(H^\times) f(X) = f(H^\times) K^\times = f(H^\times).
\]

\ref{prop:global-iso-maps-unit-group-to-unit-group(2)} 
Let $X \in \mathcal P(H^\times)$. By Remark \ref{remarks:units-and-unit-stability}\ref{remarks:units-and-unit-stability(1)}, we have $H^\times = \allowbreak XH^\times$. It follows, by part \ref{prop:global-iso-maps-unit-group-to-unit-group(1)}, that
\[
K^\times = \allowbreak f(H^\times) = \allowbreak f(X) f(H^\times) = f(X) K^\times, 
\]
which, by the same Remark \ref{remarks:units-and-unit-stability}\ref{remarks:units-and-unit-stability(1)}, shows that $f(X) \subseteq \allowbreak K^\times$. This yields $f[\mathcal P(H^\times)] \subseteq \mathcal P(K^\times)$, and the same reasoning applied to $f^{-1}$ establishes the reverse inclusion.
\end{proof}

It remains an open question whether an analogue of Theorem~\ref{thm:global-iso-maps-unit-group-to-unit-group} holds for finitary power semigroups. More explicitly, if $H$ and $K$ are monoids and $f \colon \mathcal{P}_{\fin}(H) \to \mathcal{P}_{\fin}(K)$ is an isomorphism, must $f$ restrict to an isomorphism from $\mathcal{P}_{\fin}(H^\times)$ onto $\mathcal{P}_{\fin}(K^\times)$? 

The difficulty is that, when $H$ is infinite, the unit group $H^\times$ need not belong to $\mathcal P_{\fin}(H)$. As a result, the argument used below in the proof of Theorem~\ref{thm:groups-are-P-closed} does not carry over to the finitary
setting. In particular, it cannot be used in the proof of Theorem~\ref{thm:main}, which therefore requires a completely different
(and considerably more sophisticated) approach.

\begin{proof}[Proof of Theorem \ref{thm:groups-are-P-closed}]
Let $f$ be a global isomorphism from a group $H$ to a semigroup $K$. By \cite[Lemma 1.1]{Gou-Isk-Tsi-1984}, $K$ is a monoid. It follows from Theorem \ref{thm:global-iso-maps-unit-group-to-unit-group} that $f$ restricts to a global isomorphism from $H^\times$ to $K^\times$. 
But $H$ is a group, and hence $H=H^\times$. Therefore $f$ maps
$\mathcal P(H)$ onto $\mathcal P(K^\times)$. Since, by hypothesis,
$f$ maps $\mathcal P(H)$ onto $\mathcal P(K)$, we obtain $
\mathcal P(K)=\mathcal P(K^\times)$, which is only possible if $K=K^\times$. Thus $K$ is a group, and Shafer's half-page note \cite{Sha-1967} yields $H \cong K$.
\end{proof}

\section{Preliminaries}
\label{sec:preliminaries}

In this section, we collect a number of auxiliary facts that will be used in Section \ref{sect:4}. Most notably, we prove a pair of combinatorial results that are perhaps of independent interest: one relating a natural equation in a power semigroup to the Fibonacci sequence (Proposition \ref{prop:fib-count}), and another giving a uniform bound for certain fibers of setwise multiplication (Proposition \ref{prop:fiber-counting}).

We begin with a finitary analogue of a characterization~\cite[Lemma 1.1]{Gou-Isk-Tsi-1984} of when the large power semigroup of a semigroup has an identity element. Although the argument follows similar lines, we include the details here for the sake of completeness.

\begin{lemma}\label{lem:identity-reconstruction}
Let $H$ be a semigroup. If the finitary power semigroup $\Pfin(H)$ of $H$ has an identity element $E$, then $H$
is a monoid and $E$ is the singleton $\{1_H\}$.
\end{lemma}

\begin{proof}
Let $E$ be the identity element of $\Pfin(H)$. If $u, v \in E$, then $uv \in uE = \{u\}$ and $vu \in Eu = \allowbreak \{u\}$, that is, $uv = vu = u$. By symmetry, it follows that $E$ is a singleton, say $E = \{e\}$. Hence,
\[
\{xe\} = xE = \{x\} = Ex = \{ex\},
\qquad\text{for all } x \in H.
\]
This shows that $H$ is a monoid with identity element $e$, which completes the proof since the identity element of a monoid is unique (and hence we have $e = 1_H$).
\end{proof}

The next result highlights a natural occurrence of the Fibonacci sequence in the study of power semigroups and will play a crucial role later in the proof of Theorem \ref{thm:main}.

\begin{proposition}
\label{prop:fib-count}
Let $H$ be a monoid, and suppose that $a \in H$ has infinite order.
Then, for every integer $n \ge 0$, the equation
\begin{equation}
\label{lem:fib-count:eq(1)}
\{1_H,a\} X = \{1_H,a\}^n
\end{equation}
has exactly $F_n$ solutions $X \in \mathcal P(H)$, where
$F_n$ is the $n$-th Fibonacci number (i.e., $F_0 := 0$, $F_1 := \allowbreak 1$, and $F_n := F_{n-1} + F_{n-2}$ for $n \ge 2$). Moreover, any such solution is finite and contains $1_H$.
\end{proposition}

\begin{proof}
Fix $n \in \mathbb N$, and let $\mathcal S_n$ be the set of all $X \in \mathcal P(H)$ that solve Eq.~\eqref{lem:fib-count:eq(1)}. Put $s_n := |\mathcal S_n|$. 

Every solution $X$ to Eq.~\eqref{lem:fib-count:eq(1)} is contained in $\{1_H, a\}^n = \{1_H, a, \ldots, a^n\}$, as the left-hand side of the equation contains $X$. In addition,  $1_H \in \{1_H, a\} X$ forces $1_H \in X$, by the fact that $a$ is an element of infinite order in $H$ (and hence $a^k \ne 1_H$ for all $k \in \mathbb N^+$). 

Accordingly, any solution to Eq.~\eqref{lem:fib-count:eq(1)} belongs to $\mathcal P_\fin(H)$. Moreover, if $H_a$ denotes the cyclic submonoid of $H$ generated by $a$ and $\phi_a$ is the (monoid) isomorphism $H_a \to \mathbb N$ that maps an element $x \in H_a$ to the unique integer $k \ge 0$ such that $x = a^k$, then the augmentation 
\[
\Phi_a \colon \mathcal P(H_a) \to \mathcal P(\mathbb N) \colon X \mapsto \phi_a[X] 
%\{\phi_a(x): x \in X\}
\]
of $\phi_a$ restricts to a (monoid) isomorphism from $\mathcal P_{\fin}(H_a)$ onto $\mathcal P_{\fin}(\mathbb N)$ that maps $\{1_H, a\}$ to $\{0, 1\}$. Therefore, $\Phi_a$ provides a bijection between $\mathcal S_n$ and the solutions $Y \in \mathcal P_\fin(\mathbb N)$ of the equation
\begin{equation}
\label{lem:fib-count:eq(2)}
\{0,1\} + Y = \llb 0, n \rrb.
\end{equation}
In other words, $s_n$ counts the number of sets $Y \in \mathcal P_\fin(\mathbb N)$ for which Eq.~\eqref{lem:fib-count:eq(2)} holds. In particular, we have $s_0 = 0$, because the sumset $\{0, 1\} + Y$ contains at least two elements for every non-empty $Y \subseteq \mathbb N$, whereas $\llb 0, 0 \rrb$ is a singleton. Consequently, we will assume below that $n \ge 1$. 

Let $Y$ be a solution to Eq.~\eqref{lem:fib-count:eq(2)}. We know from the first lines of the proof (now with $\mathbb N$ in place of $H$) that $0 \in Y$ and $Y\subseteq \llb 0, n \rrb$. Thus, taking maxima on both sides of the equation gives
\[
\{0,n-1\} \subseteq Y \subseteq \llb 0,n-1 \rrb.
\]
For $n=1$, this implies $Y = \{0\}$, that is, $s_1 = 1$. So, we
will henceforth assume that $n \ge 2$.

Now, fix $k \in \llb 1, n-1 \rrb$. Clearly, $k \notin Y + \{0, 1\}$ if and only if neither $k$ nor $k-1$ belongs to $Y$. For Eq.~\eqref{lem:fib-count:eq(2)} to hold, it is thus necessary and sufficient that at least one of $k-1$ and $k$
is in $Y$. E\-quiv\-a\-lent\-ly, the
complement of $Y$ in $\llb 1,n-2 \rrb$
contains no two consecutive integers. 

Conversely, if a subset $Z$ of $\llb 1,n-2 \rrb$ has no two
consecutive integers, then its complement $Z^c$ in $\llb 0,n-1 \rrb$ is a solution to Eq.~\eqref{lem:fib-count:eq(2)}; in particular, we have $\{0,n-1\} \subseteq Z^c$.

Thus $s_n$ is equal to the number of subsets of $\llb 1, n-2 \rrb$
containing no two consecutive integers. We shall call such subsets
\evid{admissible} (note that the empty set is admissible). We claim that
\[
s_n = s_{n-1} + s_{n-2}.
\]
Indeed, if an admissible subset $Z$ of $\llb 1,n-2 \rrb$ does not contain $n-2$, then $Z$ is an admissible subset of $\llb 1, n-3 \rrb$, giving $s_{n-1}$ possibilities (this also covers the cases $n = 2$ and $n = 3$, since then $\llb 1, n-3 \rrb = \emptyset$ and $s_{n-1} = 1$). 
Otherwise, if $n-2 \in Z$, then necessarily $n \ge 3$ and $n-3 \notin Z$, and so $Z \setminus \{n-2\}$ is an admissible subset of $\llb 1, n-4 \rrb$, giving $s_{n-2}$ possibilities (this also covers
the cases $n=3$ and $n=4$, since then $\llb 1,n-4 \rrb=\emptyset$ and $s_{n-2}=1$).

As we have already shown that $s_0 = 0$ and $s_1 = 1$, it follows that the 
sequence $s_0, s_1, \dots$ satisfies the same recurrence and initial conditions 
as the Fibonacci sequence $F_0, F_1, \dots$, and therefore $s_n = F_n$ for all 
$n \in \mathbb{N}$.
\end{proof}

The following definition will be useful in the next two results and will
play an essential role later in Section \ref{sect:4} in the proof of Theorem~\ref{thm:main}.

\begin{definition}
\label{def:indexed-signed-bin-rep}
An \evid{indexed signed binary representation} is any sum of the form $\sum_{i=1}^k 2^{e_i} \varepsilon_i$,
where $k \in \mathbb N$,
$e_1, \ldots, e_k \in \mathbb N$, and $\varepsilon_1, \ldots, \varepsilon_k \in \{\pm 1\} \subseteq \mathbb Z$. If the sum adds up to an integer $z$, we say it is a representation of $z$. We call $k$ the
\evid{length} and the $e_i$'s the \evid{exponents} of the representation.
\end{definition}

Note that, in contrast with the usual binary representation of an integer, the exponents
of an indexed signed binary representation need not be pairwise distinct.

We begin with a lemma that is a special case of more general finiteness results on sums of $S$-units in recurrence sequences (see, for instance, \cite[Theorem~2.1]{ber-haj-pin-rou-2019}). For the reader's convenience, we include a proof tailored to the Fibonacci sequence.

\begin{lemma}
\label{lem:fib-binary-representations}
For each $k \in \mathbb N^+$, there are only finitely many $n \in \mathbb N^+$ such that the $n$-th Fibonacci number $F_n$ has an indexed signed binary representation of length at most $k$.
\end{lemma}

\begin{proof}
Fix $k \in \mathbb N^+$, and suppose for a contradiction that the set $\mathcal N$ of positive integers $n$ such that $F_n$ has an indexed signed binary representation of length at most $k$ is infinite. 
There are only finitely many possible lengths and, for each fixed
length, only finitely many possible choices of signs. Hence,
after replacing $\mathcal N$ by an infinite subset if necessary, we may
assume (without loss of generality) that both the length and the signs are fixed as $n$ ranges over $\mathcal N$. Namely, there exists $s \in \mathbb N^+$, along with signs $\varepsilon_1, \ldots, \varepsilon_s \in \{\pm 1\}$ and functions $e_1, \ldots, e_s \colon \mathcal N \to \mathbb N$, such that
\begin{equation}
\label{lem:fib-binary-representations:eq(1)}
F_n = \sum_{i = 1}^s 2^{e_i(n)} \varepsilon_i,
\qquad\text{for all }
n \in \mathcal N.
\end{equation}

Let $n \in \mathcal N$, and put $\alpha := \frac{1}{2}(1+\sqrt5)$ and $\beta := \frac{1}{2}(1-\sqrt5)$. By Binet's formula \cite[p.~296]{Bur-2010}, we have $
\sqrt{5}\, F_n = \alpha^n - \beta^n$.
Therefore, Eq.~\eqref{lem:fib-binary-representations:eq(1)} can be equivalently restated as 
\begin{equation}
\label{lem:fib-binary-representations:eq(2)}
T_0(n) + T_1(n) + \cdots + T_{s+1}(n) = 0,
\end{equation}
where we define
\begin{equation}
\label{lem:fib-binary-representations:eq(1a)}
T_0(n) := \frac{\alpha^n}{\sqrt{5}}
\quad\text{and}\quad
T_1(n) := - \frac{\beta^n}{\sqrt{5}}, 
\end{equation}
as well as  
\begin{equation}
\label{lem:fib-binary-representations:eq(1b)}
T_{i+1}(n) := -2^{e_i(n)} \varepsilon_i, 
\qquad\text{ for each } 
i \in \llb 1, s \rrb.
\end{equation}

By Eq.~\eqref{lem:fib-binary-representations:eq(2)}, the tuple $
(T_0(n), T_1(n), \ldots, T_{s+1}(n))$
can be identified with a zero-sum sequence over the additive group of $\mathbb C$. It then follows from the basics of zero-sum theory \cite[Section 3.4]{Ger-HK-2006} that there exists at least one ``block'' $B_n \subseteq \llb 0, s + \allowbreak 1 \rrb$ with $0 \in B_n$ such that the sum $\sum_{i \in B_n} T_i(n)$ is zero and any subsum $\sum_{i \in A} T_i(n)$ with $\emptyset \subsetneq A \subsetneq B_n$ is non-zero.
Moreover, since $s$ is fixed, there are only finitely many possible choices for the block $B_n$. So, passing (if necessary) to a further infinite subset of $\mathcal N$, we may assume that $B_n$ is independent of $n$. 

To sum it up, we have established that, without loss of generality, there exists a subset $B$ of $\llb 0, s + 1 \rrb$ containing $0$ with the ``minimality property'' that
\begin{equation}
\label{lem:fib-binary-representations:eq(3)}
\sum_{i \in B} T_i(n) = 0, 
\qquad\text{for all } n \in \mathcal N,
\end{equation}
and additionally,
\begin{equation}
\label{lem:fib-binary-representations:eq(4)}
\sum_{i \in A} T_i(n) \ne 0,
\qquad \text{for all }n \in \mathcal N \text{ and every non-empty set } A \subsetneq B.
\end{equation}

Denote by $E$ the quadratic extension of the rational field obtained by adjoining $\sqrt{5}$. Since
$\beta = \allowbreak -\alpha^{-1}$, each of the summands appearing in Eq.~\eqref{lem:fib-binary-representations:eq(3)} belongs to the subgroup $\Gamma$ of $E^\times$ (the unit group of $E$) generated by $-1$, $2$, $\alpha$, and $\sqrt 5$. 
Next, note that $r := |B| - 1$ is a positive integer (by the fact that $0 \in B$ and $T_0(n) \ne 0$ for all $n \in \mathbb N$), and let $i_1, \ldots, i_r$ be an enumeration of $B \setminus \{0\}$.
We gather from the above that, for any $n \in \mathcal N$, the $r$-tuple 
\[
t_n := (-T_{i_1}(n)/T_0(n), \ldots, -T_{i_r}(n)/T_0(n))
\]
provides a non-degenerate solution to the equation 
\begin{equation}
\label{lem:fib-binary-representations:eq(5)}
x_1 + \cdots + x_r = 1, \qquad
\text{with } x_1, \ldots, x_r \in \Gamma;
\end{equation}
in particular, non-degeneracy is guaranteed by Eq.~\eqref{lem:fib-binary-representations:eq(4)}. On the other hand, Eq.~\eqref{lem:fib-binary-representations:eq(5)} has only finitely many non-degenerate solutions by Theorem~\ref{thm:ess}. Passing once again to an infinite subset of $\mathcal N$ if necessary, we may therefore assume that the tuple $t_n$ (and hence each of its components) is independent of $n$. In particular, there exists $a \in \Gamma$ such that
\begin{equation}
\label{lem:fib-binary-representations:eq(6)}
a = \frac{T_i(n)}{T_0(n)} = \frac{\sqrt{5} \: T_i(n)}{\alpha^n} = (-1)^n \sqrt{5} \, \beta^n\, T_i(n),
\qquad
\text{for all } n \in \mathcal N,
\end{equation}
where, for convenience, $i$ is the minimum of $i_1, \ldots, i_r$. We will show that this is impossible.

Indeed, fix $n \in \mathcal N$. If $i = 1$, then we get from Eqs.~\eqref{lem:fib-binary-representations:eq(1a)} and \eqref{lem:fib-binary-representations:eq(6)} that $|a| = |\beta|^{2n}$ for infinitely many $n \in \mathbb N^+$, which is absurd because $|\beta| = -\beta < 1$ and hence $|\beta|^{2n} \to 0$ as $n \to \infty$. Therefore, we must have $2 \le i \le s+1$. Accordingly, we infer from Eqs.~\eqref{lem:fib-binary-representations:eq(1b)} and \eqref{lem:fib-binary-representations:eq(6)} that 
\begin{equation}
\label{lem:fib-binary-representations:eq(7)}
0 \ne a = -\frac{2^{e_{i - 1}(n)} \varepsilon_{i-1} \sqrt{5}}{\alpha^n}.
\end{equation}
Let
$\sigma$ be the only non-trivial automorphism of the field $E$, so that $\sigma(q) = q$ for any $q \in \mathbb Q$ and $\sigma(\sqrt{5}) = -\sqrt{5}$. Then $\sigma(\alpha) = \beta$, and applying $\sigma$ to Eq.~\eqref{lem:fib-binary-representations:eq(7)} yields
\[
0 \ne \frac{\sigma(a)}{a} = - \frac{\alpha^n}{\sigma(\alpha)^n} = -\frac{\alpha^n}{\beta^n} = (-1)^{n-1} \alpha^{2n} =: \gamma_n.
\]
But this is again a contradiction, because $\alpha > 1$ and hence $|\gamma_n| \to \infty$ as $n \to \infty$.

Thus, the only possible conclusion is that there are only finitely many terms in the Fibonacci sequence with an indexed signed binary representation of bounded length.
\end{proof}

We continue with a combinatorial result that allows us to bound the number~$s$ 
of solutions $Y$ to an equation of the form $AY = B$ over a finitary 
power semigroup in terms of the length of an indexed signed binary 
representation of $s$, assuming only that $s$ is finite.

\begin{proposition}
\label{prop:fiber-counting}
Let $A$ and $B$ be non-empty finite subsets of a semigroup $H$, and let
$\mathcal S$ be the set of all $Y \in \mathcal P_{\fin}(H)$ such that
$AY=B$. If $\mathcal S$ is finite, then $|\mathcal S|$ has an indexed
signed binary representation of length at most $2^{|B|}$.
\end{proposition}

\begin{proof}
There is nothing to prove if $\mathcal S=\emptyset$, since in this case
$|\mathcal S|=0$ has an indexed signed binary representation of length
zero. Thus, assume that $\mathcal S$ is non-empty, and put
\[
\Omega := \{y \in H: Ay \subseteq B\}.
\]
Accordingly, for every $b \in B$ define
\begin{equation}
\label{prop:fiber-counting:eq(1)}
\Omega_b := \{y \in \Omega: b \in Ay\} 
\quad\text{and}\quad
E_b := \{Y \subseteq \Omega: Y \cap \Omega_b = \emptyset\}.
\end{equation}

If $Y \in \mathcal S$, then $AY = B$ and hence $Ay \subseteq B$ for each $y \in Y$, which shows that $Y \subseteq \Omega$. Moreover, if $Y \in \mathcal S$ and $y \in \Omega$, then $Ay \subseteq B = AY$ and therefore
\[
A(Y \cup \{y\}) = AY \cup Ay = B,
\]
which implies $Y \cup \{y\} \in \mathcal S$. It follows that 
\[
\Omega \subseteq \bigcup_{Y \in \mathcal S} Y, 
\]
and hence $|\Omega| < \infty$, because $\mathcal S$ is finite and each set in $\mathcal S$ is finite too (by hypothesis).

Now, let $Y \subseteq \Omega$. Since $AY \subseteq B$, we have
$Y \notin \mathcal S$ if and only if $AY \neq B$. That is, $Y \notin \mathcal S$ if and only if
there exists $b \in B$ such that $b \notin AY$, which, by our definitions, is equivalent to $Y \cap \Omega_b = \allowbreak \emptyset$ for some $b \in B$. As we have already noticed that $\mathcal S \subseteq 2^\Omega$, it follows that
\[
\mathcal S = 2^{\Omega} \setminus \bigcup_{b \in B} E_b.
\]
By the inclusion-exclusion principle, this results in
\begin{equation}
\label{prop:fiber-counting:eq(2)}
|\mathcal S| = \sum_{D\subseteq B}
(-1)^{|D|} \left|\bigcap_{b \in D} E_b \right|.
\end{equation}
If, on the other hand, $D$ is a subset of $B$ and $\Omega(D)$ is the complement of $\bigcup_{b \in D} \Omega_b$ in $\Omega$, then
\[
\bigcap_{b \in D} E_b 
= \left\{Y \subseteq \Omega: Y \cap \bigcup_{b \in D} \Omega_b = \emptyset\right\} = \left\{Y: Y \subseteq \Omega \setminus \bigcup_{b \in D} \Omega_b\right\} = 2^{\Omega(D)}.
\]

Consequently, we conclude from Eq.~\eqref{prop:fiber-counting:eq(2)} that
\[
|\mathcal S| = \sum_{D\subseteq B}
(-1)^{|D|} \left|2^{\Omega(D)}\right| = \sum_{D\subseteq B}
(-1)^{|D|} \, 2^{|\Omega(D)|},
\]
which is an indexed signed binary representation of $|\mathcal S|$ of
length at most $2^{|B|}$.
\end{proof}

We conclude the section with a lemma on intervals in totally ordered groups that will be used later in the proof of Lemma~\ref{lem:bounding-the-size-of-powers}. Although we will only apply the lemma to additive subgroups of $\mathbb Q$, the proof works in a somewhat more general setting, so
we record the result in that form.

Here and throughout, an \evid{ordered semigroup} is a pair $(H,\preceq)$ consisting of a semigroup $H$ and an order $\preceq$ on (the underlying set of) $H$ such that if $x \preceq y$ then $ux \preceq uy$ and $xu \preceq yu$ for all $u \in H$.
We say in this case that the order is \evid{translation-invariant}. If,
in addition, $\preceq$ is \evid{total} (that is, either $x \preceq y$ or $y \preceq x$ for all $x, y \in H$), we call $(H,\preceq)$ a \evid{totally ordered semigroup}.

\begin{remark}
\label{rem:minimizer}
If $\mathcal H = (H, \preceq)$ is a totally ordered semigroup, then for every non-empty finite $X \subseteq H$ there exist \textit{unique} elements $x_\ast, x^\ast \in X$ such that $x_\ast \preceq y \preceq x^\ast$ for all $y \in X$; we refer to $x_\ast$ as the \evid{$\preceq$-minimum} and to $x^\ast$ as the \evid{$\preceq$-maximum} of $X$. 
Accordingly, we have a well-defined function $\mu \colon \mathcal P_\fin(H) \to H$ that maps a set $X \in \mathcal P_\fin(H)$ to its $\preceq$-minimum. We call $\mu$ the \evid{minimizer} of $\mathcal H$.

Let $X, Y \in \mathcal P_\fin(H)$. If $x_\ast := \mu(X)$ and $y_\ast := \mu(Y)$, then $x_\ast \preceq x$ and $y_\ast \preceq y$ for all $x \in X$ and $y \in Y$, which, by translation-invariance, yields $x_\ast y_\ast \preceq xy$. It follows
that 
\[
\mu(XY) = x_\ast y_\ast = \mu(X) \mu(Y),
\]
that is, $\mu$ is a (semigroup) homomorphism $\mathcal P_\fin(H) \to H$. Note that $\mu$ is also surjective, because $\{x\} \in \mathcal P_\fin(H)$ and $\mu(\{x\}) = x$ for every $x \in H$. 
\end{remark}

A [totally] ordered semigroup $\mathcal H = (H, \preceq)$ is a [totally] ordered \textit{group} if $H$ is a group; and it is commutative/abelian if so is the semigroup $H$. For instance, every (additive) sub[semi]group of $\mathbb Q$ is a totally ordered commutative [semi]group under the usual ordering of the rationals.

\begin{lemma}
\label{lem:intervals}
Let $\mathcal H = (H, \preceq)$ be a totally ordered group. If $A$ is a non-empty subset of $H$ with minimum element $1_H$ and maximum element $a$ with respect to the order $\preceq$, then 
\[
A \cdot [1_H, b\,]_\mathcal{H} = [1_H, ab]_\mathcal{H},
\qquad\text{for all }
b \in H \text{ with } a \preceq b.
\]
\end{lemma}

\begin{proof}
Fix $b \in H$ with $a \preceq b$. We will show that $A \cdot [1_H, b\,]_\mathcal{H} \subseteq [1_H, ab]_\mathcal{H} \subseteq A \cdot [1_H, b\,]_\mathcal{H}$.

To start with, if $x \in A$ and $y \in [1_H, b]_\mathcal{H}$, then $1_H \preceq x \preceq a$ and $1_H \preceq y \preceq b$. By the
translation-invariance of the order, this implies $1_H \preceq xy \preceq ab$, and hence $xy \in [1_H, ab]_\mathcal{H}$. In other words, $A \cdot [1_H, b]_\mathcal{H} \subseteq [1_H, ab]_\mathcal{H}$ (in fact, the conclusion holds regardless of whether $a \preceq b$). 

Conversely, let $z \in [1_H, ab]_\mathcal{H}$. We need to verify that $z \in A \cdot [1_H, b]_\mathcal{H}$. 
If $z \preceq b$, then 
$1_H \in A$ yields $z = 1_H z \in A \cdot [1_H,b]_\mathcal{H}$, and we are done.
Otherwise, $b \prec z \preceq ab$. Since $a \preceq b$, we have $1_H \preceq \allowbreak a^{-1}b$, and therefore $1_H \preceq a^{-1}b \prec a^{-1}z \preceq b$.
Thus $a^{-1}z \in [1_H,b]_\mathcal{H}$, which, together with $a \in A$, gives $z = a(a^{-1}z) \in A \cdot [1_H,b]_\mathcal{H}$ and completes the proof.
\end{proof}
% -----------------------------------------------------------------------------
\section{The finitary setting}
\label{sect:4}

This section is devoted to the proof of Theorem \ref{thm:main}. Before proceeding, we record a couple of elementary properties of power semigroups that will be used later on.

\begin{remark}
\label{rem:commutativity-and-unity-are-preserved}
Let $H$ and $K$ be semigroups, and let $f \colon \mathcal{P}_{\fin}(H) \to 
\mathcal{P}_{\fin}(K)$ be an isomorphism.

It is routine to check that a 
semigroup is commutative if and only if its large power semigroup is. 
Since every subsemigroup of a commutative semigroup is itself commutative and commutativity is preserved under isomorphisms, it follows that $H$ is commutative if and only if $K$ is.

Suppose, on the other hand, that $H$ is a monoid. Then $\mathcal P_\fin(H)$ 
itself is a monoid, its identity element being the singleton $\{1_H\}$. 
It follows that $\mathcal P_\fin(K)$ is also a monoid, its identity 
being the image of $\{1_H\}$ under $f$. By Lemma~\ref{lem:identity-reconstruction}, this is only possible if $K$ is a monoid too.
\end{remark}

The next lemma shows that any finitary power-semigroup isomorphism restricts to an isomorphism between the respective unit groups, acting via singletons.

\begin{lemma}
\label{lem:unit-reconstruction}
Let $H$ and $K$ be monoids, and let $f$ be a (semigroup) isomorphism from $\mathcal P_\fin(H)$ to $\mathcal P_\fin(K)$. There exists an isomorphism $\eta: H^\times \to K^\times$ such that $f(\{u\}) = \{\eta(u)\}$ for all $u \in H^\times$.
\end{lemma}

\begin{proof}
By \cite[Theorem~4.2(i)]{Tri-Survey}, the units of
$\mathcal P_\fin(H)$ are precisely the singletons $\{u\}$ such that
$u$ is a unit of $H$, and similarly for the units of
$\mathcal P_\fin(K)$. Moreover, we have from
Remark~\ref{remarks:units-and-unit-stability}\ref{remarks:units-and-unit-stability(3)} that $f$ maps units of
$\mathcal P_\fin(H)$ bijectively onto units of
$\mathcal P_\fin(K)$. It follows that there exists a bijection $\eta \colon H^\times \to \allowbreak K^\times$ such that $f(\{u\}) = \{\eta(u)\}$ for every $u \in H^\times$. 
Since $f(\{xy\}) = f(\{x\}) f(\{y\})$ for all $x, y \in H$, it is then clear that $\eta$ is an isomorphism from $H^\times$ to $K^\times$.
\end{proof}

We are now ready to state and prove the main (technical) results for this section.

\begin{proposition}
\label{prop:structural-splitting}
Let $\mathcal H = (H, \preceq)$ be a totally ordered abelian group, $K$ be a semigroup, and $f$ be an isomorphism from $\mathcal P_\fin(H)$ to $\mathcal P_\fin(K)$. The following hold: 
\begin{enumerate}[label=\textup{(\roman{*})}]
\item\label{prop:structural-splitting(i)}
$K$ is a commutative monoid, and there exists a submonoid $L$ of $K$ such that 
\begin{equation}
\label{prop:structural-splitting:eq(0)}
L^\times = \{1_K\}
\quad\text{and}\quad
K \cong L \times K^\times \cong L \times H.
\end{equation}
\item\label{prop:structural-splitting(ii)} $\mu \circ f^{-1}(\{y\}) = 1_H$ for every $y \in L$, where $\mu$ is the minimizer of $\mathcal H$.
\end{enumerate}
\end{proposition}

\begin{proof}
We have from Remark \ref{rem:commutativity-and-unity-are-preserved} that $K$ is a commutative monoid, and from Remark \ref{rem:minimizer} that $\mu$ is a homomorphism from $\mathcal P_\fin(H)$ to $H$. In particular, $H$ being a group and $K$ being 
a monoid ensures by Lemma~\ref{lem:unit-reconstruction} that there is an isomorphism $\eta \colon H \to K^\times$ such that 
\begin{equation}
\label{prop:structural-splitting:eq(1)}
f(\{u\}) = \{\eta(u)\},
\qquad\text{for each }
u \in H.
\end{equation}

We thus obtain a well-defined function
\begin{equation}
\label{prop:structural-splitting:eq(1b)}
\lambda \colon K \to K^\times \colon y \mapsto \eta \circ \mu\circ f^{-1}(\{y\}).
\end{equation}
Notice that $\lambda$ fixes the units of $K$. Indeed, if $v \in K^\times$, then $v = \eta(u)$ for some $u \in H$, and hence
\[
\lambda(v) \stackrel{\eqref{prop:structural-splitting:eq(1)}}{=} \eta \circ \mu(\{u\}) = \eta(u) = v.
\]
In particular, $\lambda(1_K) = 1_K$. On the other hand, $\lambda$ is a 
homomorphism from $K$ to $K^\times$, as it is a composition of homomorphisms. 
Consequently, 
\begin{equation}
\label{prop:structural-splitting:eq(2)}
L := \lambda^{-1}(1_K) = \{y \in K: \lambda(y) = 1_K\}
\end{equation}
is a submonoid of $K$. Moreover, the unit group of $L$ is trivial: if $v \in L^\times$, then $v \in K^\times$ and thus $1_K = \lambda(v) = v$ (recall that $\lambda$ acts as the identity on $K^\times$). We claim that 
\[
K \cong L \times K^\times
\]
Since $H \cong \allowbreak K^\times$ (via $\eta$), this will complete the proof, as the injectivity of $\eta$ together with Eqs.~\eqref{prop:structural-splitting:eq(1b)} and \eqref{prop:structural-splitting:eq(2)} forces $\mu \circ f^{-1}(\{y\}) = 1_H$ for every $y \in L$.

For the claim, let $y \in K$ and define $\lambda_y := \lambda(y)$. Since $\lambda_y \in K^\times$, we may consider the element $y\lambda_y^{-1} \in K$. Then, $\lambda$ being a homomorphism $K \to K^\times$ that fixes $K^\times$ pointwise, we have
\[
\lambda(y \lambda_y^{-1}) = \lambda(y) \lambda(\lambda_y^{-1}) = \lambda_y \lambda_y^{-1} = 1_K.
\]
Hence, we can define a function $\phi \colon K \to L \times K^\times$ by 
\[
\phi(y) := (y \lambda_y^{-1}, \lambda_y),
\qquad \text{for each } y \in K.
\]

We claim that $\phi$ is a semigroup isomorphism. Indeed, the commutativity of $K$ entails that
\[
y\lambda_y^{-1} \cdot z \lambda_z^{-1} = yz \cdot \lambda_y^{-1} \lambda_z^{-1} = yz \cdot (\lambda_{yz})^{-1},
\qquad\text{for all } y, z \in K.
\]
It is then evident that $\phi$ is a homomorphism, and we are left to check that it is also bijective.

Suppose first that $\phi(y) = \phi(z)$ for some $y, z \in K$. Then $(y\lambda_y^{-1}, \lambda_y) = 
(z\lambda_z^{-1}, \lambda_z)$, and hence $\lambda_y = \allowbreak \lambda_z$ and $y\lambda_y^{-1} = z\lambda_z^{-1}$, which is only possible if $y = z$. So, $\phi$ is injective.

For surjectivity, let $(y, v) \in L \times K^\times$, and put $x := yv \in K$. By Eq.~\eqref{prop:structural-splitting:eq(2)} and the properties of $\lambda$, 
we have $\lambda(y) = 1_K$ and $\lambda(v) = v$. It follows that $\lambda_x = \lambda(yv) = \lambda(y)\lambda(v) = v$, and hence
\[
\phi(x) = (x \lambda_x^{-1}, \lambda_x) = (yvv^{-1}, v) = (y, v).
\]
This proves that $\phi$ is onto (and therefore an isomorphism).
\end{proof}

Incidentally, we remark that the ``structural splitting'' described by Eq.~\eqref{prop:structural-splitting:eq(0)} need not hold for arbitrary commutative monoids. For instance, let $K := \{e,u,x\}$ be a three-element set and define a binary operation on $K$ by the following Cayley table:
\[
\begin{array}{c|ccc}
\cdot & e & u & x \\
\hline
e & e & u & x \\
u & u & e & x \\
x & x & x & x
\end{array}
\]
It is routine to verify that $K$ is a commutative monoid with identity element $e$ and unit group $K^\times = \{e, u\}$. If $K$ were isomorphic to $L \times K^\times$ for some monoid $L$, then 
\[
3 = |K| = |L| \cdot |K^\times| = 2\,|L|,
\]
which is impossible. This shows that, in general, the unit group of a commutative monoid does not split off as a direct factor, a property that is, however, essential to our proof of Theorem~\ref{thm:main}.

\begin{lemma}
\label{lem:bounding-the-size-of-powers}
Let $H$ be an additive subgroup of the rational numbers, and let $f$ be an isomorphism from $\mathcal P_\fin(H)$ to the finitary power semigroup $\mathcal P_\fin(K)$ of a semigroup $K$. Then either $K$ is a group, or there exist a positive element $q \in H$ and integer $N \ge 1$ such that the size of the set $f(\{0, q\})^n$ is bounded by $N$ for infinitely many $n \in \mathbb N^+$.
\end{lemma}

\begin{proof}
The standard order $\le$ on $\mathbb Q$ turns $H$ into a totally ordered abelian group. Consequently, we are guaranteed by part \ref{prop:structural-splitting(i)} of Proposition \ref{prop:structural-splitting} that $K$ is a commutative monoid and there exists a submonoid $L$ of $K$ with trivial unit group such that $K \cong L \times K^\times$.

Now, suppose that $K$ is not a group. Then $L$ is not a group either. Accordingly, pick a non-unit $y \in L$. We claim that the (non-empty) set $A := f^{-1}(y)$ is not a singleton. Suppose the contrary. Since $H$ is a group, $A$ is then a unit of $\mathcal P_\fin(H)$. Consequently, $\{y\} = f(A)$ is a unit of $\mathcal P_\fin(K)$, because $f$ sends units to units. That is, $y$ is a unit of $L$ (a contradiction). 

On the other hand, we have from part \ref{prop:structural-splitting(ii)} of Proposition \ref{prop:structural-splitting} that the $\le$-minimum of $A$ is zero. Denote by $A^\ast$ the subgroup of $H$ generated by $A$, and set
\[
[a,b]_A := \{x \in A^\ast: a \le x \le b\},
\qquad\text{for } a, b \in A^\ast.
\]
Being a non-trivial finitely generated (additive) subgroup of $\mathbb Q$, the group $A^\ast$ is cyclic of infinite order (and hence isomorphic to $\mathbb Z$). Let $q$ be the positive generator of $A^\ast$ and $r$ be the $\le$-maximum of $A$. There then exists $m \in \mathbb N^+$ such that  $r = mq$, and we may consider the sets
\[
Q := m\{0, q\} \in \mathcal P_\fin(H)
\quad\text{and}\quad
B := f(Q) = f(\{0, q\})^m \in \mathcal P_\fin(K).
\]
To complete the proof, it is enough to check that $|B^n| \le |B|$ for all $n \in \mathbb N^+$.

To this end, fix $n \in \mathbb N^+$. Lemma \ref{lem:intervals} (applied to $A^\ast$ with the order inherited from $H$) yields
\begin{equation}
\label{lem:bounding-the-size-of-powers:eq(1)}
A + [0, nr]_A = [0, (n+1)r]_A.
\end{equation}
On the other hand, it is straightforward from our definitions that 
\[
[0, kr]_A = [0, mkq]_A = \{0, q, \ldots, mkq\} = mk \{0, q\} = kQ,
\qquad\text{for every }
k \in \mathbb N^+. 
\]
It thus follows from Eq.~\eqref{lem:bounding-the-size-of-powers:eq(1)} that $
A + nQ = (n+1) Q$, which in turn implies
\[
B^{n+1} = f( (n+1) Q) = f(A) f(nQ) = yB^n,
\]
and hence $|B^{n+1}| = |yB^n| \le |B^n|$. 
By a routine induction, it is then clear that $|B^n| \le |B|$.
\end{proof}

We finally have all the ingredients needed to prove the main result of the paper.

\begin{proof}[Proof of Theorem \ref{thm:main}]
Assume $H$ is non-trivial, or else the conclusion is obvious. Since $\mathcal P_\fin(H) \cong \mathcal P_\fin(K)$, we have from Proposition \ref{prop:structural-splitting} that $K \cong L \times H$, where $L$ is a submonoid of $K$ with trivial unit group. Suppose for a contradiction that $L \ne \{1_K\}$. 

In light of Lemma
\ref{lem:bounding-the-size-of-powers}, there exist an integer $N \ge 1$ and a non-zero element $q \in H$ such that $|f(A)^n| \le N$ for infinitely many $n \in \mathbb N^+$, where $A := \{0, q\}$. Let $s_n$ be the number of solutions $Y \in \allowbreak \mathcal P_\fin(K)$ to the equation $
BY = B^n$ as $n$ ranges over $\mathbb N^+$, where $B := f(A)$. This is the same as the number of solutions $X \in \mathcal P_\fin(H)$ to the equation
\[
\{0, q\} + X = n\{0, q\}.
\]
Indeed, since the inverse $f^{-1}$ of $f$ is an isomorphism from
$\mathcal P_\fin(K)$ to $\mathcal P_\fin(H)$, it maps solutions of the
former equation bijectively onto solutions of the latter.

It then follows from Proposition \ref{prop:fib-count} that $s_n$ equals $F_n$ (the $n$-th Fibonacci number) for all $n \in \mathbb N$; in particular, $s_n$ is finite. So, letting $k_n$ be the minimum length of an indexed signed binary representation of $F_n$, we infer from Proposition \ref{prop:fiber-counting} and the above that 
\[
k_n \le 2^{|B^n|} \le 2^{N},
\qquad\text{for infinitely many }
n \in \mathbb N^+.
\]
This is, however, impossible by
Lemma~\ref{lem:fib-binary-representations}. We must therefore conclude that
$L = \{1_K\}$, and hence $K \cong H$. In particular, $K$ is a group.
\end{proof}
\section{Prospects for future work}

We conjecture that Question~\ref{que:2}\ref{que:2(2)} has a positive answer for the class of all groups. As a first step, it would
be interesting to settle the question for torsion-free abelian groups.
By Theorem~\ref{thm:main}, the conjecture holds for groups isomorphic
to additive subgroups of $\mathbb Q$.

In a related direction, we have already noted after the proof of
Theorem~\ref{thm:global-iso-maps-unit-group-to-unit-group} that it would
be useful to know whether every isomorphism
$f \colon \mathcal P_{\fin}(H) \to \mathcal P_{\fin}(K)$ between the
finitary power semigroups of two monoids $H$ and $K$ restricts to an
isomorphism from $\mathcal P_{\fin}(H^\times)$ onto
$\mathcal P_{\fin}(K^\times)$.

% -----------------------------------------------------------------------------
\section*{Acknowledgements}

The authors were supported by the Natural Science Foun\-da\-tion of Hebei Province through grant A2023205045. 
%They are grateful to an anonymous referee for their careful reading and helpful comments on an earlier version of this manuscript.
% -----------------------------------------------------------------------------

\end{document}